\documentclass[letterpaper, 10 pt, conference]{ieeeconf}  

\IEEEoverridecommandlockouts                              
\usepackage[T1]{fontenc}
\usepackage[latin9]{inputenc}
\usepackage{amsmath}
\usepackage{amssymb}
\usepackage{graphicx}
\usepackage{dsfont}

\makeatletter

\pdfpageheight\paperheight
\pdfpagewidth\paperwidth

\@ifundefined{showcaptionsetup}{}{%
 \PassOptionsToPackage{caption=false}{subfig}}
\usepackage{subfig}
\makeatother

\newcommand{\twobytwo}[4]{\left. \begin{array}{cc} #1 & #2 \\ #3 & #4 \end{array} \right.}
\newcommand{\statespace}[4]{\left[
\begin{tabular}{c|c}
$#1$ & $#2$ \\ \hline
$#3$ & $#4$
\end{tabular}
\right]
}
\newcommand{\dist}[2]{\text{dist}\left( #1 , #2 \right)}

\newcommand{\RLinf}{\mathcal{RL}_\infty}
\newcommand{\stab}{\mathcal{RH}_\infty}
\newcommand{\Hinf}{\mathcal{H}_\infty}
\newcommand{\T}[1]{\mathcal{T}_{#1}}
\newcommand{\Th}[1]{\hat{\mathcal{T}}_{#1}}

\def\QED{~\rule[-1pt]{5pt}{5pt}\par\medskip}
\newtheorem{defn}{Definition}
\newtheorem{rem}{Remark}
\newtheorem{lem}{Lemma}

\newtheorem{prop}{Proposition}

\newtheorem{thm}{Theorem}
\newtheorem{problem}{Problem}

\newtheorem{example}{Example}

\begin{document}

\title{Distributed Control Subject to Delays Satisfying an $\mathcal{H}_\infty$ Norm Bound}

\author{Nikolai Matni
\thanks{N. Matni is with the Department of Control and Dynamical Systems, California Institute of Technology, Pasadena, CA.
 \tt{\small nmatni@caltech.edu}.}
\thanks{This research was in part supported by NSF NetSE, AFOSR, the Institute for Collaborative Biotechnologies through grant W911NF-09-0001 from the U.S. Army Research Office, and from MURIs ``Scalable, Data-Driven, and Provably-Correct Analysis of Networks'' (ONR) and ``Tools for the Analysis and Design of Complex Multi-Scale Networks'' (ARO).  The  content does not necessarily reflect the position or the policy of the Government, and no official endorsement should be inferred.}}
\maketitle
\begin{abstract}
This paper presents a characterization of distributed controllers subject to delay constraints induced by a strongly connected communication graph that achieve a prescribed closed loop $\mathcal{H}_\infty$ norm.  Inspired by the solution to the $\mathcal{H}_2$ problem subject to delays, we exploit the fact that the communication graph is strongly connected to decompose the controller into a local finite impulse response component and a global but delayed infinite impulse response component.  This allows us to reduce the control synthesis problem to a linear matrix inequality feasibility test.
\end{abstract}

\section{Introduction}

The identification of Quadratic Invariance (QI) \cite{RL06} as an appropriate condition for the convexification of structured model matching problems has brought a renewed enthusiasm and excitement to optimal controller synthesis.  In the following discussion, we survey recent results in this area, and in particular comment on three classes of quadratically invariant constraints: (1) sparsity constraints, in which we assume no delay in information sharing, but rather a restriction of what measurements each controller has access to, (2) delay constraints, in which we assume that controllers communicate with each other subject to delays induced by a strongly connected communication graph, and hence eventually have access to global, but delayed, information, and (3) delay-sparsity constraints, in which we allow both restrictions on measurement access and communication delay between controllers.

In the $\mathcal{H}_2$ case, explicit state-space solutions exist for delay-constrained \cite{LD12}, sparsity constrained \cite{SL10, SP10} and delay-sparsity constrained \cite{LL13} state-feedback problems.  When moving to the output feedback case, specific sparsity constrained problems have been solved explicitly, such as the state-space solution for the two-player case \cite{LL11}.  The delay-sparsity-constrained case has earned considerable attention, with solutions via vectorization \cite{RL06} and semi-definite programming \cite{R06,G06} existing -- we note that although computationally tractable, in contrast with the sparsity constrained setting, none of these methods claim to yield a controller of minimal order.  In the case of delay constraints without sparsity, the aforementioned results are applicable, but an additional method based on quadratic programming and spectral factorization \cite{LD14} also exists. 

The landscape of distributed $\mathcal{H}_\infty$ controller synthesis is comparably much sparser, so to speak.  However, especially in the sparsity constrained case, there has recently been some progress.  In particular, \cite{S13} provides a semi-definite programming solution for the structured optimal $\mathcal{H}_\infty$ output-feedback problem subject to nested sparsity constraints.  A more general approach, applicable to all three classes of constraint types, is presented in \cite{AR13}.  It allows for a principled approximation of the problem via a semi-definite programming based solution that computes an optimal $\mathcal{H}_\infty$ controller within a fixed finite-dimensional subspace.  By allowing this finite impulse response (FIR) approximation to be of large enough order, they are able to achieve near optimal performance in a computationally tractable manner.

This paper aims to provide a solution to the sub-optimal distributed $\mathcal{H}_\infty$ control problem subject to delay constraints -- in particular, we seek a delay constrained controller that achieves a prescribed closed loop norm.  Inspired by the results in \cite{LD14}, we exploit the fact that the controller can be written as a direct sum of a local FIR filter and a delayed, but global, infinite impulse response (IIR) element, and show that the synthesis problem can be reduced to a linear matrix inequality (LMI) feasibility test. 

A caveat is that our method is based on the so-called ``1984'' approach to $\mathcal{H}_\infty$ control, and as such, suffers from the same computational burden that the centralized solution is subject to.  We do not claim that our solution is computationally scalable, but provide it rather as evidence that in the case of delay constrained $\mathcal{H}_\infty$ synthesis, the problem admits a finite-dimensional formulation.  Our hope is that this result, much as was the case for its centralized analogue, will be a stepping stone to more computationally scalable and explicit results.

This article is organized as follows: Section \ref{sec:problem} establishes notation, and formalizes the distributed $\mathcal{H}_\infty$ model matching problem subject to delay constraints.  In Section \ref{sec:review}, we provide a refresher on the ``1984'' solution to the $\mathcal{H}_\infty$ problem, as described in \cite{F87}.  Section \ref{sec:new} provides the main result of the paper, and we demonstrate our algorithm on a three-player chain example in Section \ref{sec:example}.  We end with a discussion and conclusions in Section \ref{sec:conc}, and the Appendix contains useful formulae for computing the  transfer matrix factorizations and approximations required by our method.

\section{Problem Formulation}
In all of the following, we work in discrete-time.
\label{sec:problem}
\subsection{Notation and Operator Theoretic Preliminaries}
\label{sec:notation}
Here we establish notation and remind the reader of some standard results from operator theory, taken from \cite{F87}.
\begin{itemize}
\item $\mathcal{H}_2$ denotes the set of stable proper transfer matrices that are norm square integrable on the unit circle with vanishing negative Fourier coefficients; i.e. if $G \in \mathcal{H}_2$ then $H(z) = \sum_{i=0}^\infty H_i z^{-i}$ and $\|H\|_2^2 = \text{trace}\left(\sum_{i=0}^\infty H_i^*H_i\right)$.
\item $\mathcal{H}_\infty$ denotes the set of stable proper transfer matrices.  Note that $G \in \mathcal{H}_\infty$ implies $G \in \mathcal{H}_2$.
\item $\mathcal{L}_\infty$ denotes the frequency domain Lesbesgue space of essentially bounded functions.
\item The prefix $\mathcal{R}$ to a set $\mathcal{X}$ indicates the restriction to real-rational members of $\mathcal{X}$.
\item $\| \cdot \|_\infty$ denotes the norm on $\mathcal{L}_\infty$.
\item For $R\in \mathcal{L}_\infty$, let $\dist{R}{\Hinf}:= \inf\{\|R-X\|_\infty \, : \, X \in \Hinf\}$.
\item $\| \cdot \|$ denotes the spectral norm (maximum singular value).
\item For a transfer matrix $G\in\RLinf$, $G^\sim$ denotes its conjugate, i.e. $G^\sim(z)=G^*(z^{-1})$.
\item $\oplus$, and $\perp$, denote the direct sum, and orthogonality, respectively, as defined with respect to the standard inner product on $\mathcal{H}_2$.
\item Decompose $R\in \RLinf$ as $R= R_1^\sim + R_2$, with $R_1, \, R_2 \in \stab$, and $R_1$ strictly proper.  We shall refer to $(R_1,R_2)$ as an anti-stable/stable decomposition of $R$.
\item $\Gamma_F$ denotes the Hankel operator with symbol $F$, that is to say the Hankel mapping from $\mathcal{H}_2$ to $\mathcal{H}_2^\perp$.  Note that if $(F_1,F_2)$ is an anti-stable/stable decomposition of $F$, then $\Gamma_F = \Gamma_{F_1^\sim}$.
\item $\tilde{\Gamma}_F$ denotes the adjoint Hankel operator with symbol $F$, that is to say the Hankel mapping from $\mathcal{H}_2^\perp$ to $\mathcal{H}_2$.  The following useful fact then holds: 
\begin{equation}\|\Gamma_F\| = \|{\Gamma}_{F_1^\sim}\|=\|\tilde{\Gamma}_{F_1}\|. \label{eq:hankels} \end{equation}
\item $\Delta_N$ denotes the N-delay operator, i.e. $\Delta_N G = \frac{1}{z^N}G$.
\end{itemize}

\subsection{The model-matching problem subject to delay}
We provide a brief overview of the distributed optimal control problem subject to delay, and refer the reader to \cite{LD14} for a much more thorough and general exposition.

Let $P$ be a stable discrete-time plant given by
\begin{equation}
P=\left[\begin{array}{c|cc}
A & B_{1} & B_{2}\\
\hline C_{1} & 0 & D_{12}\\
C_{2} & D_{21} & 0
\end{array}\right]=\begin{bmatrix}P_{11} & P_{12}\\
P_{21} & P_{22}
\end{bmatrix}
\label{eqn:plant}
\end{equation}
with inputs of dimension $p_{1},\, p_{2}$ and outputs of dimension
$q_{1},\, q_{2}$. We restrict attention to stable plants for simplicity.
These methods could also be applied to an unstable plant if a stable
stabilizing nominal controller can be found, as in \cite{RL06}.  Future work will look to incorporate the results in \cite{LD14}, which are based on those in \cite{SM11}, into our procedure so as to have a general solution to the model matching problem.

Throughout, we assume that
\begin{itemize}
\item $D_{12}^{T}D_{12}>0,$ 
\item $D_{21}D_{21}^{T}>$0, 
\item $C_{1}^{T}D_{12}=0$
\item $B_{1}D_{21}^{T}=0$
\end{itemize}
so as to ensure the existence of stabilizing solutions to the necessary discrete algebraic Riccati equations (DAREs).

For $N\geq1$, define the space of $\stab$ FIR transfer matrices by $\mathcal{X}_N=\oplus_{i=0}^{N-1}\frac{1}{z^{i}}\mathbb{C}^{p_{2}\times q_{2}}$. Note that in the following, we sometimes suppress the subscript and write $\mathcal{X}_N=\mathcal{X}$ when $N$ is clear from context.

In this paper, we are concerned with controller constraints described
by delay patterns that are imposed by \emph{strongly connected communication
graphs}. As such, let  $\mathcal{S}\subset \stab$
be a subspace of the form
\begin{equation}
\mathcal{S=\mathcal{Y}}\oplus\Delta_N\stab \label{eq:constraint}
\end{equation}
where 
\begin{equation}
\mathcal{Y}=\oplus_{i=0}^{N-1}\frac{1}{z^{i}}\mathcal{Y}_{i}\subset\oplus_{i=0}^{N-1}\frac{1}{z^{i}}\mathbb{R}^{p_{2}\times q_{2}}\subset\mathcal{X}_N.
\end{equation}  

Specifically, this implies that every decision-making agent has access to \emph{all} measurements that are at least $N$ time-steps old. 

We can therefore partition the measured outputs $y$ and control inputs $u$
according to the dimension of the subsystems:
\[
y=[\begin{array}{ccc}
y_{1}^{T} & \cdots & y_{m}^{T}]^{T}\ \ u=[\begin{array}{ccc}
u_{1}^{T} & \cdots & u_{n}^{T}]^{T}\end{array}\end{array}
\]
and then further partition each constraint set $\mathcal{Y}_{i}$
as
\begin{equation}
\mathcal{Y}_{i}=\begin{bmatrix}\mathcal{Y}_{i}^{11} & \cdots & \mathcal{Y}_{i}^{1m}\\
\vdots & \ddots & \vdots\\
\mathcal{Y}_{i}^{n1} & \cdots & \mathcal{Y}_{i}^{nm}
\end{bmatrix},\label{eq:partition}
\end{equation}
where 
\begin{equation}
\mathcal{Y}_{i}^{jk}= \begin{cases} \mathbb{R}^{p_{2}^{j}\times q_{2}^{k}} &\mathrm{if} \ u_j \ \mathrm{has \ access \ to} \ y_k \ \mathrm{at \ time} \ i\\
0 & \mathrm{otherwise} 
\end{cases} 
\end{equation}
and $\sum_{j=1}^n p_2^j=p_2$, $\sum_{k=1}^m q_2^k=m$. 
\begin{figure}
\centerline{\includegraphics{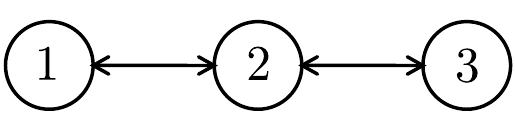}}
\caption{\footnotesize The graph depicts the communication structure of the three-player chain problem. Edge weights (not shown) indicate the delay required to transmit information between nodes.}
\label{fig:chain}
\end{figure}

\begin{example}
Consider the three player chain problem as illustrated in Figure \ref{fig:chain},
with communication delay $\tau_{c}$ between nodes. Then
\begin{equation}\begin{array}{rcl}
\mathcal{S} & = & \begin{bmatrix} \stab & \frac{1}{z^{\tau_{c}}}\stab & \frac{1}{z^{2\tau_{c}}}\stab\\
\frac{1}{z^{\tau_{c}}}\stab & \stab & \frac{1}{z^{\tau_{c}}}\stab\\
\frac{1}{z^{2\tau_{c}}}\stab & \frac{1}{z^{\tau_{c}}}\stab & \stab
\end{bmatrix}\label{eq:exS}\\
 & = & \oplus_{i=0}^{2\tau_{c}-1}\frac{1}{z^{i}}\mathcal{Y}_{i}\oplus\Delta_{2\tau_{c}}\stab
\end{array}\end{equation}
with
\begin{equation}
\mathcal{Y}_{i}=\begin{cases}
\begin{bmatrix}* & 0 & 0\\
0 & * & 0\\
0 & 0 & *
\end{bmatrix} & \mathrm{for\ }0\leq i<\tau_{c}\\
\begin{bmatrix}* & * & 0\\
* & * & *\\
0 & * & *
\end{bmatrix} & \mathrm{for\ }\tau_{c}\leq i < 2\tau_{c}
\end{cases},
\label{eq:3chain}
\end{equation}
where, for compactness, {*} is used to denote a space of appropriately
sized real matrices.  In this setting, every decision maker then has access to all measurements that are at least $2\tau_c$ time-steps old.

\end{example}

The distributed control problem of interest is to design a controller
$K\in\mathcal{S}$ so as to achieve a pre-defined closed loop $\mathcal{H}_{\infty}$
norm.  Specifically, the problem is to find an internally stabilizing $K \in \mathcal{S}$ such that
\begin{equation}
||P_{11}+P_{12}K(I-P_{22}K)^{-1}P_{21}||_{\infty} \leq \gamma
\label{eq:dec_problem}
\end{equation}
for some pre-defined $\gamma>\gamma_{\inf}$, where $\gamma_{\inf}$ is the infimum achievable closed loop $\Hinf$ norm.

In \cite{RL06}, it was shown that to be able to pass to the Youla parameter
$Q=K(I-P_{22}K)^{-1}$ without loss in (\ref{eq:dec_problem}), the constraint
set must be  \emph{quadratically invariant}.
\begin{defn}
A set $\mathcal{S}$ is \emph{quadratically invariant} under
$P_{22}$ if
\[
KP_{22}K\in\mathcal{S\ \mathrm{for\ all\ }}K\in\mathcal{S}
\]
\end{defn}
In the case of delay-constraints imposed by a communication graph, intuitive and easily verifiable conditions for QI can be stated \cite{RCL10}.  Essentially these conditions say that in order to have QI, controllers must be able to communicate with each other faster than dynamics propagate through the plant -- this is closely related to partial nestedness \cite{HC72} and poset causality \cite{SP10}.
 
Thus, if QI holds, the feasibility problem (\ref{eq:dec_problem})
can be reduced to the following model matching problem:
\begin{problem}
Find $Q\in\mathcal{S\,\bigcap}\stab$ such that
\begin{equation}
 ||\T{1}+\T{2}Q\T{3}||_{\infty} \leq \gamma
\end{equation}
for some $\gamma > \gamma_{\inf}$, with $\T{1}=P_{11}$, $\T{2} = P_{12}$ and $\T{3}=P_{21}$.
\end{problem}

\section{A Review of ``1984'' $\mathcal{H}_\infty$ Control}
\label{sec:review}
As our solution is based on the so-called ``1984'' approach to $\mathcal{H}_\infty$ control, we review it in this section.  The following is based on material found in \cite{F87}.

\subsection{$\T{3} = I$ Case}
\label{sec:identity}
 We begin with the solution to the sub-optimal model matching problem with $\T{3}=I$ first, as the general case follows from a nearly identical derivation.  Specifically, we consider the problem:
\begin{problem}
Find $Q\in\stab$ such that $\|\T{1}-\T{2}Q\|_\infty\leq \gamma$ for some $\gamma>\gamma_{\inf}\geq 0$, where $\gamma_{\inf}$ is the optimal achievable closed-loop $\Hinf$ norm.
\end{problem}

In order to state the main result, we first define the following transfer matrices:
\begin{enumerate}
\item Let $U_i, \, U_o$ be an inner-outer factorization of $\T{2}$ such that $\T{2}=U_iU_o$, with $U_i^\sim U_i=I$, and $U_i,\, U_o,\, U_o^{-1} \in \stab$.
\item Let $Y:=(I-U_iU_i^\sim)\T{1}$.
\item For $\gamma>\|Y\|_\infty$, let $Y_o$ be a bi-stable spectral factor of $\gamma^2I - Y^\sim Y$ such that $\gamma^2I - Y^\sim Y=Y_o^\sim Y_o$, with $Y_o, \, Y_o^{-1} \in \stab$.
\item Define the $\RLinf$ matrix $R:=U_i^\sim\T{1}Y_o^{-1}$.
\end{enumerate}

\begin{thm}
Let $\alpha := \inf \{ \|\T{1}-\T{2}Q\|_\infty \ : \ Q \in \stab \}$.

Then 
\begin{enumerate}
\item $\alpha = \inf\{ \gamma \ : \ \|Y\|_\infty < \gamma, \ \dist{R}{\stab}<1 \}$, and
\item For $\gamma > \alpha$ and $Q,\,X\in \stab$ such that
\begin{itemize}
\item $\|R-X\|_\infty\leq 1$, and
\item $X = U_oQY_o^{-1}$,
\end{itemize}
 we have that $\|\T{1}-\T{2}Q\|_\infty \leq \gamma$.
 \end{enumerate}
\label{thm:classic}
\end{thm}

Before proving this result, we need the following two preliminary lemmas:

\begin{lem}
Let $U$ be inner and $E\in \RLinf$ be given by
\[ E := \begin{bmatrix} U^\sim \\  I - UU^\sim \end{bmatrix}. \]
Then for all $G\in\RLinf$, we have that $\|EG\|_\infty = \|G\|_\infty$.
\label{lem:helper1}
\end{lem}

\begin{lem}
For $F,\,G\in\RLinf$ with the same number of columns, if
\begin{equation} \left\| \begin{bmatrix} F \\ G \end{bmatrix} \right\|_\infty < \gamma \label{eq:helper2}\end{equation}
then $\|G\|_\infty < \gamma$ and $\|FG_o^{-1}\|_\infty <1$, where $G_o$ is a bi-stable spectral factor of $\gamma^2 I -G^\sim G$.

Conversely, if $\|G\|_\infty < \gamma$ and $\|FG_o^{-1}\|_\infty \leq 1$, then \eqref{eq:helper2} holds.
\label{lem:helper2}
\end{lem}

\begin{lem}[{Nehari's Theorem}]
For any $R\in\RLinf$, we have that \[ \dist{R}{\stab}=\dist{R}{\mathcal{H}_\infty}=\|\Gamma_R\|, \] and that there exists $X\in\stab$ such that $\|R-X\|_\infty=\dist{R}{\stab}$.
\label{lem:nehari}
\end{lem}

We may now prove Theorem \ref{thm:classic}.

\begin{proof}

1) Let $\gamma_{\inf}:= \inf\{\gamma \ : \ \|Y\|_\infty < \gamma, \ \dist{R}{\stab}<1 \}$.

 Choose $\epsilon >0$ such that $\alpha < \gamma < \alpha + \epsilon$, implying that there exists $Q\in\stab$ such that $\|\T{1}-\T{2}Q\|_\infty < \gamma$.  Then, by Lemma \ref{lem:helper1}, we have that
 \begin{equation}
 \left\| \begin{bmatrix} U_i^\sim \\  I - U_iU_i^\sim \end{bmatrix}(\T{1}-\T{2}Q)\right\|_\infty < \gamma.
 \label{eq:cp1}
 \end{equation}
 Now, notice that 
  \begin{equation}
 \begin{bmatrix} U_i^\sim \\  I - U_iU_i^\sim \end{bmatrix}\T{2}=\begin{bmatrix} U_o \\ 0 \end{bmatrix},
 \label{eq:cp2}
 \end{equation}
making \eqref{eq:cp1} equivalent to
\begin{equation}
 \left\| \begin{bmatrix} U_i^\sim\T{1}-U_oQ \\  Y \end{bmatrix}\right\|_\infty < \gamma.
 \label{eq:cp3}
\end{equation}

Applying Lemma \ref{lem:helper2}, this then implies that
\begin{equation}
\|Y\|_\infty <\gamma,
\label{eq:cp4}
\end{equation}
and
\begin{equation}
\|U_i^\sim T_1 Y_o^{-1} - U_oQY_o^{-1}\|_\infty <1
\label{eq:cp5}
\end{equation}
By Lemma \ref{lem:nehari}, this in turn implies that $\dist{R}{U_o(\stab) Y_o^{-1}}<1$, 
which, noting that $U_o$ is right invertible in $\stab$ and that $Y_o$ is invertible in $\stab$, is equivalent to
\begin{equation}
\dist{R}{\stab}<1
\label{eq:cp6}
\end{equation}
Then, from \eqref{eq:cp4} and \eqref{eq:cp6}, and the definition of $\gamma_{\inf}$ we conclude that $\gamma_{\inf}\leq \gamma$, and thus that $\gamma < \alpha + \epsilon$.  Since $\epsilon$ was arbitrary, we then have that $\gamma_{\inf} \leq \alpha$.

To prove the reverse inequality, again choose $\epsilon>0$ and $\gamma$ such that $\gamma_{\inf} < \gamma < \gamma_{\inf} + \epsilon$.  Then \eqref{eq:cp4} and \eqref{eq:cp6} hold, so \eqref{eq:cp5} holds for some $Q\in\stab$.  Applying the converse of Lemma \ref{lem:helper2}, this in turn implies that
\begin{equation}
 \left\| \begin{bmatrix} U_i^\sim\T{1}-U_oQ \\  Y \end{bmatrix}\right\|_\infty \leq \gamma.
 \label{eq:cp7}
\end{equation}
Finally, reversing the above steps, this leads to $\|\T{1}-\T{2}Q\|_\infty \leq \gamma$.  Thus $\alpha \leq \gamma < \gamma_{\inf} + \epsilon$, and hence $\alpha \leq \gamma_{\inf}$.

2)  This follows immediately from the previous derivation.
\end{proof}

Thus, a high level outline for computing an $\Hinf$ controller satisfying a $\gamma$ bound in closed loop is
\begin{enumerate}
\item Compute $Y$ and $\|Y\|_\infty$.
\item Select a trial value $\gamma > \|Y\|_\infty$.
\item Compute $R$ and $\|\Gamma_R\|$.  Then $\|\Gamma_R\|<1$ if and only if $\alpha < \gamma$, so increase or decrease $\gamma$ accordingly, and return to step 2 until a sufficiently accurate upper bound for $\alpha$ is obtained.
\item Find a matrix $X\in\stab$ such that $\|R-X\|_\infty \leq 1$.
\item Solve $X=U_oQY_o^{-1}$ for a $Q \in \stab$ satisfying $\|\T{1}-\T{2}Q\|_\infty \leq \gamma$.
\end{enumerate}

\subsection{General $\T{3}$}
\label{sec:general}
We now state the result for general $\T{3}$.  First, define the following matrices
\begin{enumerate}
\item Let $U_i, \, U_o$ be an inner-outer factorization of $\T{2}$ such that $\T{2}=U_iU_o$, with $U_i^\sim U_i=I$, and $U_i,\, U_o,\, U_o^{-1} \in \stab$.
\item Let $Y:=(I-U_iU_i^\sim)\T{1}$.
\item For $\gamma>\|Y\|_\infty$, let $Y_o$ be a bi-stable spectral factor of $\gamma^2I - Y^\sim Y$ such that $\gamma^2I - Y^\sim Y=Y_o^\sim Y_o$, with $Y_o, \, Y_o^{-1} \in \stab$.
\item Let $V_{co},\ V_{ci}$ be a co-inner-outer factorization of $\T{3}Y_o^{-1}$ such that $\T{3}Y_o^{-1} = V_{co}V_{ci}$ and $V_{ci}, \, V_{co}, \, V_{co}^{-1} \in \stab$.
\item Let $Z:=U_i^\sim\T{1}Y_o^{-1}\left(I - V_{ci}^\sim V_{ci}\right).$
\item If $\|Z\|_\infty <1$, let $Z_{co}$ be a bi-stable co-spectral factor of $I - ZZ^\sim$ such that $I-ZZ^\sim = Z_{co}Z_{co}^\sim$, with $Z_{co}, \, Z_{co}^{-1} \in \stab.$
\item Let $R := Z_{co}^{-1}U_i^\sim \T{1} Y_o^{-1}V_{ci}^\sim.$
\end{enumerate}
\begin{rem}
Notice that $R,\, Y,\, Z \in \RLinf$, and that $Z_{co}^{-1}U_o$ is right-invertible over $\stab$ and $V_{co}$ is left-invertible over $\stab$.  
\end{rem}

\begin{thm}
Let $\alpha := \inf \{ \|\T{1}-\T{2}Q\T{3}\|_\infty  : Q \in \stab \}$.

Then 
\begin{enumerate}
\item $\alpha = \inf\{ \gamma \,:\,  \|Y\|_\infty < \gamma, \, \|Z\|_\infty <1,$  $\dist{R}{\stab}<1 \}$, and
\item For $\gamma > \alpha$ and $Q,\,X\in \stab$ such that
\begin{itemize}
\item $\|R-X\|_\infty\leq 1$, and
\item $X = Z_{co}^{-1}U_oQV_{co}$,
\end{itemize}
 we have that $\|\T{1}-\T{2}Q\T{3}\|_\infty \leq \gamma$.
 \end{enumerate}
\label{thm:classic2}
\end{thm}
\begin{proof}
Analogous from that of Theorem \ref{thm:classic}, and therefore omitted.
\end{proof}

Similarly, we may outline a general high level algorithm for computing a controller using Theorem \ref{thm:classic2}:

\begin{enumerate}
\item Compute $Y$ and $\|Y\|_\infty$.
\item Select a trial value $\gamma > \|Y\|_\infty$.
\item Compute $Z$ and $\|Z\|_\infty$.
\item If $\|Z\|_\infty <1$, continue; if not, increase $\gamma$ and return to step 3.
\item Compute $R$ and $\|\Gamma_R\|$.  Then $\|\Gamma_R\|<1$ if and only if $\alpha < \gamma$, so increase or decrease $\gamma$ accordingly, and return to step 3 until a sufficiently accurate upper bound for $\alpha$ is obtained.
\item Find a matrix $X\in\stab$ such that $\|R-X\|_\infty \leq 1$.
\item Solve $X=Z_{co}^{-1}U_oQV_{co}$ for a $Q \in \stab$ satisfying $\|\T{1}-\T{2}Q\T{3}\|_\infty \leq \gamma$.
\end{enumerate}

\section{Distributed $\mathcal{H}_\infty$ Control Subject to Delays}
\label{sec:new}
As in \cite{LD14}, we exploit the fact that the communication graph is strongly connected to decompose $Q$ into a local distributed FIR filter $V\in\mathcal{Y}$ and a global, but delayed, IIR component $\Delta_N D\in \frac{1}{z^N}\stab$, where in particular, $D\in \stab$ is unconstrained:

\begin{equation}
Q = V + \Delta D \in \mathcal{S}, \ \text{with } V \in \mathcal{Y}, \ D \in \stab
\label{eq:Qdecomp}
\end{equation}

\subsection{$\T{3}=I$ Case}
\label{sec:special_dist}
We begin with a solution to the $\T{3}=I$ case to simplify the exposition, as the general case, much as in the centralized problem, follows from an analogous argument.

Let
\begin{itemize}
\item $\Th{1}(V):=\T{1}-\T{2}V$, 
\item $\Th{2}:=\T{2}\Delta_N$,
\item $\hat{U}_i:= U_i \Delta_N$, $\hat{U}_o=U_o \in \stab$ be inner and outer, respectively, such that $\Th{2}=\hat{U}_i\hat{U}_o$, and $\hat{U}_o^{-1}  \in \stab$.
\item $\hat{R}(V):=\Delta_N^\sim R-\hat{U}_o(\Delta_N^\sim V)Y_o^{-1} $,
\end{itemize}
with $Y_o^{-1}$ and $R$ defined as in Section \ref{sec:identity}.  We then have that

\begin{thm}
Let $\alpha := \inf \{ \|\Th{1}-\Th{2}D\|_\infty \ : \ D \in \stab, \ V \in \mathcal{Y} \}$.

Then 
\begin{enumerate}
\item $\alpha = \inf\{ \gamma \ : \ \|Y\|_\infty < \gamma, \ \dist{\hat{R}}{\stab}<1 \}$, and
\item For $\gamma > \alpha$ and $D,\,X\in \stab$ such that
\begin{itemize}
\item $\|\hat{R}-X\|_\infty\leq 1$, and
\item $X = \hat{U}_oDY_o^{-1}$,
\end{itemize}
 we have that $\|\Th{1}-\Th{2}D\|_\infty \leq \gamma$.
 \end{enumerate}
\label{thm:new}
\end{thm}

\begin{rem}
Although this seems to be a nearly identical restatement of Theorem \ref{thm:classic}, this in fact not true, as the factorization matrix $Y_o$ is \emph{identical} to the centralized case, and \emph{independent} of $V$.
\end{rem}

Before proving this result, we will need the following lemma:

\begin{lem}
For $\hat{Y}(V):=(I-\hat{U}_i\hat{U}_i^\sim)\Th{1}(V)$, we have that $\hat{Y}(V)=Y$, where $Y$ is as defined in Section \ref{sec:identity}.
\label{lem:new2}
\end{lem}
\begin{proof}
Noting that 
\[(I-\hat{U}_i\hat{U}_i^\sim)(\T{2}\Delta_N) =(I-\hat{U}_i\hat{U}_i^\sim)\hat{\T{2}} = 0,\] it follows immediately that 
\[ \begin{array}{rcl}
\hat{Y}(V)&=& (I-\hat{U}_i\hat{U}_i^\sim)\hat{\T{1}}\\
&=&(I-\hat{U}_i\hat{U}_i^\sim){\T{1}}-(I-\hat{U}_i\hat{U}_i^\sim)(\T{2}\Delta_N) \Delta_N^\sim V \\
&=& (I-\hat{U}_i\hat{U}_i^\sim){\T{1}}
\end{array}
\]. 

Finally, noting that $\hat{U}_i = U_i \Delta_N$, we obtain that $\hat{Y} = Y$.
\end{proof}

We may now prove Theorem \ref{thm:new}.

\begin{proof}

1) We proceed as in the proof of Theorem \ref{thm:classic}, and premultiply by
\begin{equation}
\begin{bmatrix} \hat U_i^\sim \\ (I - \hat U_i \hat U_i^\sim) \end{bmatrix},
\end{equation}
and apply Lemma \ref{lem:helper2} to obtain the equivalence between
$\|\Th{1}-\Th{2}D\|_\infty \leq \gamma$ and 
\begin{equation}
 \left\| \begin{bmatrix} (\hat U_i^\sim\Th{1}-\hat U_oD) \\  \hat{Y}(V) \end{bmatrix}\right\|_\infty < \gamma.
 \label{eq:np1}
\end{equation}

By Lemma \ref{lem:helper2} and Lemma \ref{lem:new2}, \eqref{eq:np1} is equivalent to
\begin{equation}
\|Y\|_\infty < \gamma
\label{eq:np2}
\end{equation}
and
\begin{equation}
\|\hat U_i^\sim\Th{1} Y_o^{-1}-\hat U_oDY_o^{-1}\|_\infty <1.
\label{eq:np3}
\end{equation}
Noting that 
\begin{equation}
\begin{array}{rcl}
\hat U_i^\sim\Th{1} Y_o^{-1} &=& 
\hat U_i^\sim \T{1} Y_o^{-1} - \hat U_i^\sim (\T{2}\Delta_N)\Delta_N^\sim V Y_o^{-1} \\ &=&
 \Delta_N^\sim R - \hat U_o\Delta_N^\sim V Y_o^{-1} \\
&=& \hat R(V)
\end{array}
\label{eq:np4}
\end{equation}
this is then equivalent to
\begin{equation}
\| \hat R(V) -\hat U_oDY_o^{-1}\|_\infty <1,
\end{equation}
which by the arguments of the proof of Theorem \ref{thm:classic}, is equivalent to $\|\Gamma_{\hat{R}(V)}\|<1$.

The rest of the proof proceeds as that of Theorem \ref{thm:classic}. 
\end{proof}

Thus, for a fixed $\gamma$, we have reduced the problem to a feasibility test: does there exist a FIR filter $V\in\mathcal{Y}$ such that $\|\Gamma_{\hat{R}(V)}\|<1$.  As per identity \eqref{eq:hankels}, this is equivalent to $\|\tilde{\Gamma}_{\hat R_1}\|<1$, with $(\hat R_1, \hat R_2)$ an anti-stable/stable decomposition of $\hat{R}$. 

With this in mind, let $R_1$ and $R_2$ be an anti-stable/stable decomposition of $\Delta_N^\sim R$.  Now, define $G(V) \in \stab$ as 
\begin{equation}
\begin{array}{rcl}
G(V)&:=&\hat U_o V Y_o^{-1}  \\ &=& \sum_{i=0}^\infty \frac{1}{z^i}G_i(V).
\end{array}
\label{eq:GV}
\end{equation}
 where the terms $G_i(V)$ are the impulse response elements of $G$.  It is easily verified that these terms are \emph{affine} in $\{V_i\}$, the impulse response elements of $V$ (i.e. $V=\sum_{i=0}^{N-1} \frac{1}{z^i}V_i$).  Note that $G(V) \in \stab$ follows from $U_o,\, V, \, Y_o^{-1} \in \stab$.  As such, let 
\[
G(V) :=  \statespace{A_G}{B_G}{C_G}{D_G}\]
 be a minimal stable realization of $G$.

 We then have that
 \begin{equation}
 \begin{array}{rcl}
 \hat U_o\Delta_N^\sim V Y_o^{-1} &=& \Delta_N^\sim G \\
 &=& z^N\sum_{i=0}^\infty \frac{1}{z^i}G_i(V) \\
 &=& \sum_{k=1}^{N} z^k G_{N-k}(V) + \sum_{j=0}^\infty \frac{1}{z^j}G_{j+N} \\
 &=:& q(V)^\sim + G_N(V).
 \end{array}
 \end{equation}
 with $q(V) = \sum_{k=1}^{N} \frac{1}{z^k} G^\top_{N-k}(V)  \in \stab$ and strictly proper.
 
 Also note that $G_N$ has the following state space representation
 \begin{equation}
 \begin{array}{rcl}
 G_N(V) &=& \statespace{A_G}{B_G}{C_G A_G^N}{C_G A_G^{N-1} B_G}\\ 
 &=& \statespace{A_G}{A_G^N B_G}{C_G }{C_G A_G^{N-1} B_G},
 \end{array}
 \end{equation}
 and is therefore also clearly in $\stab$.

The following lemma is an immediate consequence of the previous discussion.

\begin{lem}
Let $\hat R(V)$ be as defined.  Then an anti-stable/stable decomposition of $\hat{R}(V)$ is given by
\begin{equation}
\begin{array}{rcl}
\hat{R}_1 = R_1 - q(V) \\
\hat{R}_2 = R_2 - G_N(V)
\end{array}
\label{eq:antistabdecomp}
\end{equation}
\end{lem} 

From our previous discussion, we have thus reduced the problem to finding an FIR filter $V$ such that $\|\tilde{\Gamma}_{\hat{R}_1}\|<1$, for $\hat{R}_1$ given as in \eqref{eq:antistabdecomp}.  

We begin by deriving a state space representation for $\hat{R}_1$, and then use this representation to formulate the Hankel norm bound test as a linear matrix inequality (LMI).

First note that $q(V)$ is simply a strictly causal FIR filter, and thus has a state space representation given by
\begin{equation}
q(V) = \statespace{A_q}{B_q}{C_q(V)}{0},
\end{equation}
where $A_q$ is the down-shift operator (i.e. a block matrix with appropriately dimensioned Identity matrices along the first sub block diagonal, and zeros elsewhere), $B_q = [I,\,0\,\dots,\,0]^\top$, and $C_q(V)=[G_{N-1}(V)^\top,\dots,G_0(V)^\top]$.  Note that only $C_q(V)$ is a function of our design variable $V$.

Letting the strictly proper $R_1\in\stab$ have a minimal stable realization
\begin{equation}
R_1 = \statespace{A_r}{B_r}{C_r}{0}
\end{equation}
we then have the following realization for $\hat{R}_1\in\stab$:

\begin{equation}
\hat{R}_1 = \statespace{\twobytwo{A_r}{0}{0}{A_q}}{\begin{matrix} B_r \\ B_q \end{matrix}}{\begin{matrix} C_r & -C_q(V) \end{matrix}}{0} =: \statespace{A_R}{B_R}{C_R(V)}{0}.
\label{eq:realization}
\end{equation}

We emphasize again that our design variable $V$ appears only in $C_R(V)$. 

We now remind the reader of the variational formulation for the Hankel norm of a strictly proper transfer matrix $F \in \stab$. 

\begin{prop}
For a system \[F=\statespace{A}{B}{C}{0} \in\stab,\] we have that $\|\tilde{\Gamma}_F\|<1$ if and only if there exist matrices $P,\, Q\geq 0$ and scalar  $\lambda \geq 0$ such that
\begin{equation}
\begin{array}{rcl}
\begin{bmatrix} A^\top Q A - Q & C^\top \\
C & -\lambda I \end{bmatrix} &\leq& 0 \\
\\
\begin{bmatrix} 
-P & PB & PA \\
B^\top P & - I & 0 \\
A^\top P & 0 & - P
\end{bmatrix} & \leq & 0 \\
P - Q &\geq& 0 \\
\lambda < 1
\end{array}
\label{eq:hankel}
\end{equation}
\label{prop:hankel_lmi}
\end{prop}
\begin{proof}
This is the discrete-time analog of the variational formulation found in Section 6.3.1 of \cite{B94}.
\end{proof}

Substituting our realization \eqref{eq:realization} into (\ref{eq:hankel}), we see that this is an LMI in the variables $\{V_i\}_{i=0}^{N-1}$, $P$, $Q$, and $\lambda$, and is feasible if and only if there exists an FIR filter $V \in \mathcal{Y}$ such that $\|\tilde{\Gamma}_{\hat{R}_1}\|<1$.

Thus, a high level outline for computing a distributed controller satisfying an $\Hinf$ norm bound of $\gamma$ in closed loop is
\begin{enumerate}
\item Compute $Y$ and $\|Y\|_\infty$.
\item Select a trial value $\gamma > \|Y\|_\infty$.
\item Construct $\hat{R}_1(V)$ and check if the LMI
\begin{equation}
\begin{array}{rcl}
\begin{bmatrix} A_R^\top Q A_R - Q & C_R(V)^\top \\
C_R(V) & -\lambda I \end{bmatrix} &\leq& 0 \\
\\
\begin{bmatrix} 
-P & PB_R & PA_R \\
B_R^\top P & - I & 0 \\
A_R^\top P & 0 & - P
\end{bmatrix} & \leq & 0 \\
P - Q &\geq& 0 \\
\lambda < 1
\end{array}
\label{eq:hankel_specialized}
\end{equation}
is feasible.  This LMI is feasible if and only if $\|\tilde{\Gamma}_{\hat R_1}\|<1$, which in turn occurs if and only if $\alpha < \gamma$, so increase or decrease $\gamma$ accordingly.  This feasibility test will additionally yield an FIR filter $V$ that satisfies this bound.
\item Find a matrix $X(V)\in\stab$, dependent on $V$, such that $\|\hat R(V)-X(V)\|_\infty \leq 1$ (such a matrix is guaranteed to exist by the same arguments as those used in the centralized case).
\item Solve $X(V)= \hat U_oDY_o^{-1}$ for $D(V) \in \stab$ satisfying $\|\Th{1}-\Th{2}D\|_\infty \leq \gamma$.
\item Set $Q = V + \Delta_N D(V)\in \mathcal{S}\bigcap \stab$
\end{enumerate}

\subsection{General $\T{3}$}

Define the following transfer matrices
\begin{enumerate}
\item $\hat{Z} = \hat{U}_i^\sim \T{1}Y_o^{-1}(I-V_{ci}V_{ci}^\sim),$
\item $\hat{R} := \Delta_N^\sim R - Z_{co}^{-1}\hat{U}_o(\Delta_N^\sim V)V_{co},$
\end{enumerate}
and let $Y_o^{-1}$, $R$, $V_{co}$ and $Z_{co}^{-1}$ be as defined in Section \ref{sec:general}, and $\Th{1}, \ \Th{2}, \ \hat{U}_i$ and $\hat{U}_o$ be as defined in Section \ref{sec:special_dist}.  We note that just as $\hat{Y}(V)$ was independent of $V$, so too would be the analogous $\hat{Z}(V)$ -- as such we simply define $\hat{Z}$ and not $\hat{Z}(V)$.

\begin{rem}
Although initially surprising, the independence of $\hat{Y}(V)$ and $\hat{Z}(V)$ from $V$ is in fact fairly intuitive.  The $\mathcal{L}_\infty$ norms of $Y$ and $\hat{Z}$ correspond to fundamental performance limits as imposed by the plant, and as such should not be affected by rewriting the controller as a sum of two components, rather than as a single element.
\end{rem}

\begin{thm}
Let $\alpha := \inf \{ \|\Th{1}-\Th{2}D\T{3}\|_\infty  : D \in \stab \}$.

Then 
\begin{enumerate}
\item $\alpha = \inf\{ \gamma \ : \ \|Y\|_\infty < \gamma, \ \|Z\|_\infty <1, \  \dist{R}{\stab}<1 \}$, and
\item For $\gamma > \alpha$ and $D,\,X\in \stab$ such that
\begin{itemize}
\item $\|\hat{R}-X\|_\infty\leq 1$, and
\item $X = Z_{co}^{-1}\hat U_oDV_{co}$,
\end{itemize}
 we have that $\|\Th{1}-\Th{2}D\T{3}\|_\infty \leq \gamma$.
 \end{enumerate}
\label{thm:new2}
\end{thm}
\begin{proof}
Analogous to that of Theorem \ref{thm:new}, and therefore omitted.
\end{proof}

Just as in the $\T{3}=I$ case, this problem has now been reduced to finding an FIR filter $V\in\mathcal{Y}$ such that $\|\tilde{\Gamma}_{\hat{R}}\|<1$.  The arguments of the preceding section apply nearly verbatim, with the exception of replacing equation \eqref{eq:GV} with
\begin{equation}
 G(V) := Z_{co}^{-1}\hat{U}_o V V_{co}
 \label{eq:GV2}
 \end{equation}
 
Therefore, a high level outline for computing a distributed controller satisfying an $\Hinf$ norm bound of $\gamma$ in closed loop is
\begin{enumerate}
\item Compute $Y$ and $\|Y\|_\infty$.
\item Select a trial value $\gamma > \|Y\|_\infty$.
\item Compute $\hat Z$ and $\|\hat Z\|_\infty$.
\item  If $\|\hat Z\|_\infty <1$, continue; if not, increase $\gamma$ and return to step 3.
\item Construct $\hat{R}_1(V)$, with $G(V)$ defind as in \eqref{eq:GV2}, and check if the LMI \eqref{eq:hankel_specialized} is feasible.  This LMI is feasible if and only if $\|\tilde{\Gamma}_{\hat R_1}\|<1$, which in turn occurs if and only if $\alpha < \gamma$, so increase or decrease $\gamma$ accordingly. This feasibility test will additionally yield an FIR filter $V$ that satisfies this bound.
\item Find a matrix $X(V)\in\stab$, dependent on $V$, such that $\|\hat R(V)-X(V)\|_\infty \leq 1$.
\item Solve $X(V)= Z_{co}^{-1}\hat U_oDV_{co}$ for $D(V) \in \stab$ satisfying $\|\Th{1}-\Th{2}D\T{3}\|_\infty \leq \gamma$.
\item Set $Q = V + \Delta_N D(V) \in \mathcal{S}\bigcap \stab$
\end{enumerate}

For the convenience of the reader, we provide explicit state-space formulae for the factorizations and approximations required to implement this algorithm in the Appendix.

\section{Example}
\label{sec:example}
We consider first the full information problem ($P_{21}=I$) of a three-player chain with communication delay of $\tau_c = 1$ -- the sparsity constraint $\mathcal{Y}$ on the FIR filter is as given in equation \eqref{eq:3chain}.  The dynamics of $P_{11},\, P_{12}$ and $P_{22}$ are given by 
\begin{equation}
\begin{array}{lll}
A = \begin{bmatrix} .5 & .2 & 0 \\
				.2 & .5 & .2 \\
				0 & .2 & .5 \end{bmatrix}, &
B_1 = \left[ I_{3\times 3} \ 0_{3\times 3}\right] & B_2 = I_{3\times 3},  \\
C_1 = \begin{bmatrix} I_{3\times 3}  \\ 0_{3\times 3} \end{bmatrix}, & D_{11} = 0_{6\times 6}, & D_{12} =\begin{bmatrix} 0_{3\times 3}  \\ I_{3\times 3} \end{bmatrix}   \\
C_2 = I_{3\times 3}, &  D_{21} = \left[  0_{3\times 3} \ I_{3\times 3}  \right], & D_{22} = 0_{3\times 3},
\end{array}
\label{eq:params}
\end{equation}
and we set $P_{21}=I$.  Note that this is a suitably modified version of the output feedback problem considered in \cite{LD14}.

We first computed the optimal centralized norm of the system using classical results \cite{ZDG96}, and obtained a centralized closed loop norm of .9772.  We note that this is the theoretical lower bound as given by $\|Y\|_\infty$ from the algorithms we described above.  To verify the consistency of our algorithm, we used our LMI formulation to compute a centralized controller as well.  This was done by allowing the elements of the FIR filter $V_0$ and $V_1$ to be unconstrained, and not suprisingly, we were also able to achieve a closed loop norm of .9772 in this manner.  We then constrained $V$ to lie in the subspace $\mathcal{Y}$ as given by \eqref{eq:3chain}, and surprisingly, we were still able to achieve a closed loop norm of .9772.  This is a significant improvement over the delayed system (i.e.$V_0$ and $V_1$ constrained to be zero), for which we were only able to achieve a closed loop norm of 1.6856.

We then considered the general output-feedback problem, with $P_{21}$ given by the parameters in \eqref{eq:params} as well.  The centralized and LMI computed centralized closed loop norms were both found to be 1.502, with the best distributed norm found to be 1.515.  Once again, we see near identical performance from the centralized and distributed solutions, whereas the delayed controller was only able to achieve a closed loop norm of 2.213.

\section{Conclusion}
\label{sec:conc}
This paper presented an LMI based characterization of the sub-optimal delay-constrained distributed $\mathcal{H}_\infty$ control problem.  By exploiting the strongly connected nature of the communication graph, we were able to reduce the problem to a feasibility test in terms of the Hankel norm of a certain transfer matrix that is a function of the localized FIR component of the controller.  We note that much as in the $\mathcal{H}_2$ case, by reducing the control synthesis problem to one that is convex in the FIR filter, communication delay co-design \cite{M_CDC13_codesign} and augmentation \cite{MD_dual13} methods are applicable .  However, although finite dimensional, this method is based on the ``1984'' approach to $\mathcal{H}_\infty$ control -- as such, the computational burden is quite high, limiting the scalability of the approach.  

Future work will therefore focus on the following three aspects: (1) adapting the parameterization used in \cite{LD14} so as to relax the assumption of a stable plant, (2) formally integrating communication delay co-design methods into the controller synthesis procedure, and most pressingly (3) seeking more direct and computationally scalable means of identifying appropriate FIR filters.

\section{Acknowledgements}
The author would like to thank Andrew Lamperski for the suggestion to look at the ``old-school'' $\mathcal{H}_\infty$ literature, and John C. Doyle for pointers to references on discrete time Nehari problems.  The author would especially like to thank Seungil You for his enthusiastic revisions and unrelenting questioning of the technical content of this manuscript -- these undoubtedly improved the quality and accuracy of this work immensely.
 
\bibliographystyle{/Users/nmatni/Documents/Publications/ACC13_duality/IEEEtran}
\bibliography{hinf,/Users/nmatni/Documents/Publications/biblio/comms,/Users/nmatni/Documents/Publications/biblio/decentralized,/Users/nmatni/Documents/Publications/biblio/matni}

\begin{appendix}
In all of the following, we assume that the conditions needed for the existence of the required stabilizing solution of the corresponding Discrete Algebraic Riccati Equations (DARE) are met -- the reader is referred to \cite{ZDG96} and \cite{H93} for more details.  All ``co-X'' factorizations, where ``X'' may be either inner-outer or bi-stable spectral,  can be obtained by transposing the ``X'' factorization of the transpose system.

\subsection{Inner-Outer Factorizations}

Let 
\[ G := \statespace{A}{B}{C}{D} \in \stab. \]  From \cite{ZDG96}, an inner-outer factorization $G=U_i U_o$ of $G$, with $U_i$ inner and $U_o$ outer, is given by
\begin{equation}
U_i = \statespace{A+BF}{BH^{-1}}{C + DF}{DH^{-1}}
\end{equation}
\begin{equation}
U_o = \statespace{A}{B}{-HF}{H}
\end{equation}
with $H = (D^\top D + B^\top X B)^{\frac{1}{2}}$, and $X$ the stabilizing solution of the following DARE
\begin{equation}
\begin{array}{rcl}
X &=& A^\top X A + C^\top C + A^\top X B F, \\
F &=& - (D^\top D + B^\top X B)^{-1}B^\top X A.
\end{array}
\end{equation}

\subsection{Bi-stable Spectral Factorizations}
Let $Y \in \stab$ be strictly proper, and let
\[ G_Y = \statespace{A_Y}{B_Y}{C_Y}{0} \] be a state-space realization of the strictly proper $\stab$ component of $Y^\sim Y$.  

If $A_Y$ is invertible, then it holds that
\[
\gamma^2 I - Y^\sim Y =  G_Y + G_Y^\sim + D_Y + D_Y^\top
\]
where $D_Y = \frac{1}{2}\left( \gamma^2 I + B_Y^\top A_Y^{-\top}C_Y^\top\right).$

A bi-stable spectral factorization $\gamma^2 I - Y^\sim Y =  M^\sim M$, with $M, \, M^{-1} \in \stab$ is then given by
\begin{equation}
M = \statespace{A_Y}{B_Y}{H^{-1}(C_Y + B_Y^\top X A_Y)}{H}
\end{equation}
with $H=(D_Y + D_Y^\top + B_Y^\top X B_Y)^\frac{1}{2}$, and $X$ the stabilizing solution of the following DARE
\begin{equation}
\begin{array}{rcl}
X &=& A_Y^\top X A_Y + (A_Y^\top X B_Y + C_Y^\top) F, \\
F &=& - (D_Y^\top  + D_Y + B_Y^\top X B_Y)^{-1}(B_Y^\top X A_Y + C_Y).
\end{array}
\end{equation}

This result follows directly from standard results on spectral factors and positive real systems \cite{ZDG96}

\subsection{Stable Approximations}
The following is taken from \cite{H93}.  Let
\[G := \statespace{A}{B}{C}{D} \in \stab \] be a minimal state-space representation, and assume that $\rho = \|\tilde{\Gamma}_G\| < \gamma$.  Let $X$ and $Y$ be the controllability and observability Gramians of $G$, respectively. 

 Let $Q \in \stab$ have the state-space representation
\[ Q := \statespace{A_Q}{B_Q}{C_Q}{D_Q} \] with

\[\begin{array}{rcl}
A_Q &=& A- BC_Q \\
B_Q &=& AXC^\top + BE^\top \\
C_Q &=& (E^\top C + B^\top YA)N \\
D_Q &=& D^\top - E^\top,
\end{array}
\]

where $N=(\gamma^2 I - XY)^{-1}$, and for any unitary matrix $U$,
\begin{multline*}
E = - (I+CNXC^\top)^{-1}CNXA^\top YB \\ + \gamma(I+CNXC^\top)^{-\frac{1}{2}}U(I+B^\top YNB)^{-\frac{1}{2}} .\end{multline*}

Then $\|G-Q^\sim\|_\infty = \gamma$ and $(G-Q^\sim)^\sim (G-Q^\sim) = \gamma^2 I $.
\end{appendix}
\end{document}